\documentclass[12pt]{amsart}
\usepackage{graphicx,xcolor,bm,amsthm,amsmath,amssymb}
\usepackage[all]{xy}
\usepackage{multicol}
\usepackage{booktabs}
\usepackage{array}
\usepackage{caption}
\usepackage[a4paper,margin=1in]{geometry}
\usepackage{setspace}
\usepackage{geometry}
\definecolor{db}{RGB}{23,20,219}
\definecolor{dg}{RGB}{2,101,15}
\colorlet{sectitlecolor}{red!60!black}
\colorlet{sectboxcolor}{cyan!30}
\colorlet{secnumcolor}{orange}
\usepackage[explicit]{titlesec}
\titleformat{\section}
{\sffamily\color{sectitlecolor}\Large\bfseries\filcenter}{}{2em}{\thesection.\quad #1}%
\newtheoremstyle{mytheorem}{5pt}{}{\color{db}}{}{\color{db}\bfseries}{}{ }{}
\theoremstyle{mytheorem}
\newtheorem{theorem}{Theorem}[section]
\newtheorem{corollary}[theorem]{Corollary}

\newtheorem{proposition}[theorem]{Proposition}
\theoremstyle{definition}

\theoremstyle{example}

\theoremstyle{remark}
\newtheorem{remark}[theorem]{Remark}
\numberwithin{equation}{section}
\def\QSym{{\rm QSym}}%
\def\HQ{{\mathfrak Q}}%
 \DeclareSymbolFont{Shuffle}{U}{shuf}{m}{n}
\DeclareFontFamily{U}{shuf}{}
\DeclareFontShape{U}{shuf}{m}{n}{%
  <-8>shuffle7%
  <8->shuffle10%
}{}
\DeclareMathSymbol\shuf{\mathbin}{Shuffle}{"001}
\DeclareMathSymbol\cshuf{\mathbin}{Shuffle}{"002}

\newcommand{\Rowt}[3]{%
  \noindent
  \begin{tabular}{@{}p{0.05\linewidth}p{0.41\linewidth}p{0.47\linewidth}@{}}%
    \textbf{#1} & {\footnotesize #2} & {\footnotesize #3}\\%
  \end{tabular}%
  \par\vspace{2pt}%
}

\newcommand{\TableHeadert}{%
  \noindent
  \begin{tabular}{@{}p{0.05\linewidth}p{0.41\linewidth}p{0.47\linewidth}@{}}%
    \textbf{$n$} & \textbf{\scriptsize Composition} & \textbf{\scriptsize Value}\\%
    \midrule
  \end{tabular}%
  \par\vspace{4pt}%
}

\renewcommand{\Row}[3]{%
  \noindent
  \begin{tabular}{@{}p{0.05\linewidth}p{0.26\linewidth}p{0.63\linewidth}@{}}%
    \textbf{#1} & {\footnotesize #2} & #3\\%
  \end{tabular}%
  \par\vspace{2pt}%
}

\newcommand{\TableHeader}{%
  \noindent
  \begin{tabular}{@{}p{0.05\linewidth}p{0.26\linewidth}p{0.63\linewidth}@{}}%
    \textbf{$n$} & \textbf{\scriptsize Composition} & \textbf{\scriptsize Value}\\%
    \midrule
  \end{tabular}%
  \par\vspace{4pt}%
}
\begin{document}
\title{Coincidence among sum formulas for zeta-like multiple values}
\author{Kwang-Wu Chen}%
\address{Department of Mathmeatics, University of Taipei, Taipei 100234, Taiwan}
\email{kwchen@uTaipei.edu.tw}%
\date{\today. \quad Version:490. \quad {\it This project was initiated on} July 19, 2024} 
\begin{abstract}
We study two families of zeta-like multiple series—the multiple $\rho$-values and the multiple $\eta$-values—
defined by nested sums with shifted denominators.
An explicit factorial formula for $\rho$ reveals its intrinsic combinatorial structure
and leads to closed expressions for fixed weight and depth.
A remarkable identity
\[
\sum_{\substack{a_1+\cdots+a_{q+1}=n\\ a_i\ge0}}
\frac{1}{\prod_{j=1}^{q}(a_j+a_{j+1}+\cdots+a_{q+1}+1)}=1,
\]
emerges from a weighted-sum transformation, exhibiting a perfect discrete balance.

The main theorem proves that
\[
\sum_{\substack{|\bm{s}|=q+r+1\\ \ell(\bm{s})=r+1}}\eta(\bm{s})
=\sum_{\substack{|\bm{\alpha}|=q+r+1\\ \ell(\bm{\alpha})=q+1}}\rho(\bm{\alpha}),
\]
showing that the total sums of $\rho$- and $\eta$-values coincide
for equal weight but complementary depths.
This correspondence provides an analytic basis for integral representations of $\eta$-values
and for deriving weighted sum relations.
Together, these results show that the $\rho$- and $\eta$-families
form two complementary realizations of a unified analytic–combinatorial structure,
bridging factorial and harmonic formulations in zeta-like multiple sums.
\end{abstract}

\maketitle

\section{Introduction}

For a given positive integer \(r\), we define the multiple zeta values (MZVs) by  
\[
\zeta(\alpha_1,\ldots,\alpha_r)
=\sum_{1\le n_1<n_2<\cdots<n_r}
\frac{1}{n_1^{\alpha_1}n_2^{\alpha_2}\cdots n_r^{\alpha_r}},
\]
where \(\alpha_1,\alpha_2,\ldots,\alpha_r\) are positive integers and 
\(\alpha_r \ge 2\) to ensure convergence.  
We refer to such convergent multi-indices 
\(\bm\alpha = (\alpha_1,\ldots,\alpha_r)\) 
as {admissible indices} (that is, indices for which the defining series converges).  
The quantity \( |\bm\alpha| = \alpha_1 + \alpha_2 + \cdots + \alpha_r \) 
is called the \textbf{weight}, 
and \(\ell(\bm\alpha) = r\) is called the \textbf{depth} of \(\zeta(\bm\alpha)\).
Throughout the paper, we denote by \(|\bm\alpha|\) and \(\ell(\bm\alpha)\) 
the weight and the depth of a multi-index \(\bm\alpha\), respectively.
We use the shorthand notation \(\{a\}^k\) for \(k\) repetitions of \(a\); 
for instance, \(\zeta(\{2\}^3,5) = \zeta(2,2,2,5)\).

Multiple zeta values (MZVs) occupy a central position in modern number theory and mathematical physics  
\cite{BroadhurstK1997, Drinfeld1990, Zagier1992}.  
Their algebraic and combinatorial structures—particularly the {sum formulas}
\[
  \sum_{|\bm\alpha| = m\atop \ell(\bm\alpha)=r} \zeta(\bm\alpha) = \zeta(m),
\]
where the summation is taken over all admissible indices 
\(\bm\alpha = (\alpha_1, \ldots, \alpha_r)\) with \(r \ge 1\) 
and \(m \ge 2\) being integers, together with the 
{duality} and {shuffle–stuffle relations}, 
have been extensively studied since the pioneering works of Hoffman, Ohno, and Zagier 
\cite{Granville1997, Hoffman1997, Ohno1999, ShenJ2017, Zagier1992}.
Various analytic generalizations, such as the Hurwitz, Tornheim, and
Mordell–Tornheim type series, have further enriched the landscape of
zeta-like multiple sums \cite{Chen2024, Mordell1958, Tornheim1950}.

\medskip
In this paper, we introduce two new families of ``zeta-like'' series.  
The first one, called the \textbf{multiple \(\eta\)-value}, is defined by
\begin{equation}\label{eq.def-eta}
  \eta(\alpha_1+1,\alpha_2+1,\ldots,\alpha_r+1)
  := \sum_{n=1}^\infty 
  \frac{1}{n^{\alpha_1+1}(n+1)^{\alpha_2+1}\cdots (n+r-1)^{\alpha_r+1}},
\end{equation}
where \(\alpha_1, \alpha_2, \ldots, \alpha_r\) are nonnegative integers 
and the total weight \( |\bm{\alpha}| + r \) is assumed to be greater than \(1\) 
to ensure convergence.  

The second one, referred to as the \textbf{multiple \(\rho\)-value} 
and denoted by \(\rho(\alpha_1+1,\alpha_2+1,\ldots,\alpha_r+1)\), is defined by
\begin{equation}\label{eq.def-rho}
  \sum_{1 \le n_1 < n_2 < \cdots < n_r}
  \frac{1}{
  (n_1)_{\alpha_1+1}\,
  (n_2+\alpha_1)_{\alpha_2+1}\,
  (n_3+\alpha_1+\alpha_2)_{\alpha_3+1}\cdots
  (n_r+\alpha_1+\cdots+\alpha_{r-1})_{\alpha_r+1}},
\end{equation}
where the rising factorial is defined by 
\((x)_0 = 1\) and \((x)_n = x(x+1)\cdots(x+n-1)\) for \(n \in \mathbb{N}\).  
Here, \(\alpha_r \ge 1\) is required to guarantee convergence.

\medskip
Although their definitions differ substantially in appearance,
both $\rho$ and $\eta$ share a remarkable common feature:
their total summations, when organized by weight, exhibit
unexpected coincidences.
Through an integral transformation,
certain $\eta$-sums can be expressed exactly in terms of $\rho$-sums
with matching total weight, revealing a deep combinatorial duality
between these two classes.

Our main theorem establishes the following identity, 
which reveals a deep correspondence between the two zeta-like families 
(see Theorem~\ref{thm:rho-eta-connection}):
\[
\sum_{\substack{|\bm{s}| = q + r + 1 \\[2pt] \ell(\bm{s}) = r + 1}}
\eta(\bm{s})
\;=\;
\sum_{\substack{|\bm{\alpha}| = q + r + 1 \\[2pt] \ell(\bm{\alpha}) = q + 1}}
\rho(\bm{\alpha}),
\]
where both summations are taken over all admissible indices.  
This fundamental relation forms one of the principal results of the paper, 
establishing the analytic correspondence between the 
$\rho$- and $\eta$-families.  
In addition to this connection, we also derive several independent properties 
of each family, including explicit formulas, combinatorial identities, 
and weighted sum relations.

\medskip

The principal findings of this paper may be summarized as follows:

\begin{itemize}

  \item[(i)] An explicit factorial formula for the $\rho$-values 
  (see Proposition~\ref{prop.rho-value}) is obtained, 
  giving a closed expression in terms of factorial products 
  and providing the foundation for subsequent summation formulas.

  \item[(ii)] A general sum relation connecting $\rho$-values with
  finite multiple zeta-star values (see Theorem~\ref{thm:rho-sum}):
  \[
  \sum_{|\bm s|=m}\rho(s_1+1,\ldots,s_r+2)
  =\frac{\zeta_{m+1}^\star(\{1\}^{r-1})}{(m+1)(m+1)!},
  \]
  where
  \[
  \zeta_{m+1}^\star(\{1\}^{r-1})
  =\sum_{1\le k_1\le k_2\le\cdots\le k_{r-1}\le m+1}
  \frac{1}{k_1k_2\cdots k_{r-1}}
  \]
  denotes the finite multiple zeta-star value.
  This identity reveals the factorial structure inherent in $\rho$ 
  and serves as a discrete analogue of multiple harmonic relations.

  \item[(iii)] Several combinatorial consequences and weighted sum identities 
  arise from the study of $\rho$-values, 
  including a remarkable combinatorial equality valid for all $q,n\ge0$ 
  (see Equation~\eqref{eq.qn-any}), 
  which emerges as a by-product of a weighted sum transformation of~$\rho$.

  \item[(iv)] For nonnegative integers $n$ and $q$, 
the sum of $\eta$-values admits the following integral representation:
  (see Corollary~\ref{cor.34}):
  \begin{equation}\label{eq.int}
  \sum_{|\bm{\alpha}| = q+1}
  \eta(\alpha_1+1, \ldots, \alpha_{n+1}+1)
  =\frac{1}{n!\,q!}
  \int_{0 < t_1 < t_2 < 1}
  \!\!\left(\log \frac{t_2}{t_1}\right)^q
  (1-t_1)^n
  \frac{dt_1\,dt_2}{(1-t_1)t_2}.
  \end{equation}
  This analytic expression illustrates a continuous analogue 
  of the discrete summation structure 
  and provides a powerful bridge between factorial and harmonic frameworks.
  
  \item[(v)] Further weighted sum relations for $\eta$-values 
  are derived through integral transformations (see Proposition~\ref{prop.35}),
  revealing how these zeta-like families share a unified 
  combinatorial–analytic structure.

\end{itemize}

\medskip

To save space, we shall write $\sum_{|\bm\alpha|=m}$ 
to denote $\sum_{\alpha_1+\cdots+\alpha_r=m\atop \alpha_1,\ldots,\alpha_r\ge0}$.
This convention will be used throughout the paper when specifying
the parameters of summations.

\medskip

Together, these results reveal that the $\rho$- and $\eta$-families 
represent two complementary manifestations of a single 
analytic–combinatorial structure.  
The factorial formulas and discrete sum relations of $\rho$ 
mirror the harmonic and integral representations of $\eta$, 
while their coincidence in total weighted sums 
highlights a hidden symmetry among zeta-like multiple series.  
This correspondence not only unifies the factorial and harmonic frameworks 
but also establishes a broader analytic link between 
combinatorial identities and integral expressions.

\medskip

In the following sections, we study the two families in parallel.  
Section~2 focuses on the multiple $\rho$-functions, deriving their factorial formulas, 
combinatorial identities, and related integral forms.  
Section~3 treats the multiple $\eta$-functions, establishing their integral correspondence with $\rho$ 
and deriving explicit and weighted sum formulas.  
Finally, Section~4 offers concluding remarks and discusses possible extensions of the results.

\section{Multiple $\rho$-values}

We introduce a class of multiple series denoted by $\rho$, whose values are rational combinations of reciprocal integers. 
Although the $\rho$-function is of independent interest, its evaluation will play a key role in expressing certain $\eta$-sums, 
which can be linked to $\rho$ through integral transformations.

Let $\bm{\alpha}=(\alpha_1,\alpha_2,\ldots,\alpha_r)$ be an $r$-tuple of nonnegative integers with $\alpha_r \geq 1$. 
We define the multiple $\rho$-value, as defined in Equation~\eqref{eq.def-rho}.

A simple instance occurs when $r=1$, which reduces to a classical evaluation \cite{Boya2008, Boya2023}:
\begin{align*}
\rho(s+1)
&=\sum_{n=1}^{\infty}\frac1{(n)_{s+1}}
 =\sum_{n=1}^{\infty}\frac1{n(n+1)(n+2)\cdots (n+s)}\\
&=\frac1s\sum_{n=1}^{\infty}
 \left[\frac1{n(n+1)\cdots (n+s-1)}
       -\frac1{(n+1)(n+2)\cdots (n+s)}\right]
 =\frac1{s\cdot s!}.
\end{align*}
Hence, the simplest case yields the closed form $\rho(s+1)=1/(s\cdot s!)$.

We now establish an explicit evaluation formula for $\rho(\alpha_1+1,\alpha_2+1,\ldots,\alpha_r+1)$, 
which will serve as the basis for subsequent structural identities.

\begin{proposition}\label{prop.rho-value}
Given an $r$-tuple $\bm{\alpha}=(\alpha_1,\alpha_2,\ldots,\alpha_r)$ of nonnegative integers
with $\alpha_r\geq 1$, we have
\begin{equation}\label{eq.001}
\rho(\alpha_1+1,\alpha_2+1,\ldots,\alpha_r+1)
=\frac1{|\bm\alpha|!\cdot\prod^r_{k=1}(\alpha_k+\alpha_{k+1}+\cdots+\alpha_r)}.
\end{equation}
\end{proposition}
\begin{proof}
We begin by applying the telescoping method to the innermost summation.
For the last variable $n_r$, we have
\[
\sum_{n_r=n_{r-1}+1}^{\infty}
  \frac{1}{(n_r+\alpha_1+\alpha_2+\cdots+\alpha_{r-1})_{\alpha_r+1}}
=\frac1{\alpha_r}\frac1{(n_{r-1}+\alpha_1+\alpha_2+\cdots+\alpha_{r-1}+1)_{\alpha_r}}.
\]
Substituting this into the next layer of summation yields
\begin{align*}
&\frac{1}{\alpha_r}
\sum_{n_{r-1}=n_{r-2}+1}^{\infty}
  \frac{1}{(n_{r-1}+\alpha_1+\cdots+\alpha_{r-2})_{\alpha_{r-1}+\alpha_r+1}}\\
&\qquad= \frac{1}{\alpha_r(\alpha_{r-1}+\alpha_r)}
    \frac{1}{(n_{r-2}+\alpha_1+\cdots+\alpha_{r-2}+1)_{\alpha_{r-1}+\alpha_r}}.
\end{align*}

Proceeding iteratively in the same manner for the remaining summation variables,
we obtain the closed form
\begin{equation}\label{eq.rho-eval}
\rho(\alpha_1+1,\alpha_2+1,\ldots,\alpha_r+1)
 = \frac{1}{|\bm{\alpha}|!\,\prod_{k=1}^r(\alpha_k+\alpha_{k+1}+\cdots+\alpha_r)}.
\end{equation}
This completes the proof.
\end{proof}

Motivated by the Drinfeld integral representation of multiple zeta values \cite{Drinfeld1990}, 
we can express the multiple Pochhammer zeta values 
$\rho(\alpha_1+1,\alpha_2+1,\ldots,\alpha_r+1)$ 
in an analogous integral form. Specifically, one has
\begin{equation}\label{eq.1.2}
\int_{0<t_1<t_2<\cdots<t_{|\bm{\alpha}|+r}<1}
\Omega_1\Omega_2\cdots\Omega_{|\bm{\alpha}|+r},
\end{equation}
where for each $j$ we define
\[
\Omega_j =
\begin{cases}
\dfrac{dt_j}{1-t_j}, & \text{if } j=1,\,\alpha_1+2,\,\alpha_1+\alpha_2+3,\ldots,\,
\alpha_1+\alpha_2+\cdots+\alpha_{r-1}+r,\\[6pt]
dt_j, & \text{otherwise.}
\end{cases}
\]
Consequently, the function $\rho$ admits the following compact integral representation over the $r$-simplex:
\begin{align}\label{eq.rho-int}
&\rho(\alpha_1+1,\alpha_2+1,\ldots,\alpha_r+1) \\
&\qquad=\frac1{\alpha_1!\alpha_2!\cdots\alpha_r!}\int_{E_r}
 (t_2-t_1)^{\alpha_1}
 \cdots (t_r-t_{r-1})^{\alpha_{r-1}}(1-t_r)^{\alpha_r}
 \frac{dt_1\cdots dt_r}{(1-t_1)\cdots(1-t_r)}, \nonumber
\end{align}
where 
\[
E_r=\{(t_1,t_2,\ldots,t_r)\in(0,1)^r \mid 0<t_1<t_2<\cdots<t_r<1\}.
\]

For later use, we recall several auxiliary functions and notations.
The \emph{generalized harmonic function} is defined by
\[
H_n^{(s)}(x)=\sum_{k=1}^n \frac{1}{(k+x)^s},
\]
and in particular $H_n^{(s)}(0)=H_n^{(s)}=\sum_{k=1}^n k^{-s}$.
Given a sequence $(t_i)_{i\ge1}$, the \emph{modified Bell polynomials}
(or the cycle index of the symmetric group) $P_m(t_1,t_2,\ldots,t_m)$ 
are defined by \cite{Chen2017}
\[
\exp\!\left(\sum_{k=1}^{\infty}\frac{t_k}{k}z^k\right)
=\sum_{m=0}^{\infty} P_m(t_1,t_2,\ldots,t_m)z^m.
\]
The explicit expression for $P_m$ is given by
\[
P_m(t_1,t_2,\ldots,t_m)
=\sum_{\substack{k_1+2k_2+\cdots+mk_m=m\\ k_i\ge0}}
\frac{1}{k_1!k_2!\cdots k_m!}
\left(\frac{t_1}{1}\right)^{k_1}
\left(\frac{t_2}{2}\right)^{k_2}\cdots
\left(\frac{t_m}{m}\right)^{k_m}.
\]

For subsequent analysis, we introduce the truncated Hurwitz-type multiple zeta-star function.
For integers $m,n\ge1$ and a complex parameter $s$, define
\[
\zeta_n^{\star}(\{1\}^m; s)
 := \sum_{1 \le k_1 \le k_2 \le \cdots \le k_m \le n}
    \frac{1}{(k_1+s)(k_2+s)\cdots (k_m+s)}.
\]
When $s=0$, this reduces to the finite multiple zeta-star value
\[
\zeta_n^{\star}(\{1\}^m)
 = \sum_{1 \le k_1 \le k_2 \le \cdots \le k_m \le n}
   \frac{1}{k_1 k_2 \cdots k_m}.
\]

Let $h_m$ and $p_m$ denote, respectively, the $m$-th complete homogeneous 
and power-sum symmetric polynomials in infinitely many variables $x_i$,
where $x_k=k^{-1}$ for $1\le k\le m+1$ and $x_k=0$ otherwise. 
As shown in \cite{Chen2017}, the quantities $p_m$ and $h_m$ satisfy
\begin{align}\label{eq.p}
p_m &= H_n^{(m)}=\sum_{j=1}^n \frac{1}{j^m},\\
h_m &= \zeta_n^{\star}(\{1\}^m)
   = P_m(H_n^{(1)},H_n^{(2)},\ldots,H_n^{(m)}).
\label{eq.h}
\end{align}

We now establish a closed formula connecting multiple $\rho$-values with truncated multiple zeta-star sums.
This identity may be regarded as a finite analogue of the classical sum formula for multiple zeta values.

\begin{theorem}\label{thm:rho-sum}
For any nonnegative integer $m$ and integer $r\ge 1$, we have
\begin{align}\label{eq:rho-sum}
\sum_{|\bm{s}|=m}\rho(s_1+1,&\ldots,s_{r-1}+1,s_r+2)
 = \frac{\zeta_{m+1}^{\star}(\{1\}^{r-1})}{(m+1)(m+1)!} \\
 &= \frac{1}{(m+1)(m+1)!}
    P_{r-1}\bigl(H_{m+1}^{(1)},H_{m+1}^{(2)},\ldots,H_{m+1}^{(r-1)}\bigr).
\nonumber
\end{align}
\end{theorem}

\begin{proof}
Define
\[
A := \sum_{|\bm{s}|=m}\rho(s_1+1,s_2+1,\ldots,s_{r-1}+1,s_r+2).
\]
By formula~\eqref{eq.001}, it follows that
\[
A
 = \sum_{|\bm{s}|=m}
   \left[\,(|\bm{s}|+1)!\,
          \prod_{k=1}^r (s_k+s_{k+1}+\cdots+s_r+1)\,\right]^{-1}.
\]
Let $t_i = s_{r+1-i}$ for $1\le i\le r$.
Then the above summation becomes
\[
A = \frac{1}{(m+1)(m+1)!}
    \sum_{|\bm{t}|=m}
    \left[\prod_{j=1}^{r-1}(1+t_1+t_2+\cdots+t_j)\right]^{-1}.
\]
Define
\[
n_1 = 1+t_1,\quad
n_2 = 1+t_1+t_2,\quad
\ldots,\quad
n_{r-1} = 1+t_1+t_2+\cdots+t_{r-1}.
\]
Then we can rewrite the expression as
\[
A
 = \frac{1}{(m+1)(m+1)!}
   \sum_{1\le n_1\le n_2\le\cdots\le n_{r-1}\le m+1}
   \frac{1}{n_1n_2\cdots n_{r-1}}.
\]
The assertion follows from Equation~\eqref{eq.h}.
\end{proof}

Applying the same method as in the proof of Theorem~\ref{thm:rho-sum}, 
namely substituting the explicit value of $\rho$ from Equation~\eqref{eq.001} 
and performing a suitable change of variables, 
we obtain the following corollary.

\begin{corollary}\label{cor:rho-general}
For nonnegative integers $r$, $s$, and $q$, we have
\begin{align*}
&\sum_{|\bm{\alpha}|=r}
  \rho(\alpha_1+1,\ldots,\alpha_q+1,\alpha_{q+1}+s+2) \\
&\qquad= \frac{1}{(r+s+1)!}
   \sum_{|\bm{\alpha}|=r}
   \frac{1}{\displaystyle\prod_{k=1}^{q+1}\!
      \bigl(s+1+\alpha_k+\alpha_{k+1}+\cdots+\alpha_{q+1}\bigr)} \\[4pt]
&\qquad= \frac{1}{(r+s+1)\,(r+s+1)!}\,
   \zeta_{r+1}^{\star}(\{1\}^q; s),
\end{align*}
where $\zeta_{r+1}^{\star}(\{1\}^q; s)$ denotes the truncated Hurwitz-type 
multiple zeta-star function defined earlier.
\end{corollary}

By setting $r+s=n$ and rearranging terms, we further obtain a weighted sum formula
\[
\sum_{|\bm{\alpha}|=n}
  \rho(\alpha_1+1,\ldots,\alpha_q+1,\alpha_{q+1}+2)\,(\alpha_{q+1}+1)
   = \frac{1}{(n+1)!}.
\]
Substituting the explicit expression of \(\rho\) from Equation~\eqref{eq.001}, 
we obtain the following purely combinatorial identity.
\begin{proposition}
For all nonnegative integers \(q\) and \(n\), we have
\begin{equation}\label{eq.qn-any}
\sum_{\substack{a_1+\cdots+a_{q+1}=n\\[2pt] a_i\ge0}}
\frac{1}{
\displaystyle\prod_{j=1}^{q}
\bigl(a_j+a_{j+1}+\cdots+a_{q+1}+1\bigr)
}
=1,
\end{equation}
where, as usual, the empty product (for \(q=0\)) is taken to be~1.
\end{proposition}

Here, the variables $(a_1,\ldots,a_{q+1})$ are a simple relabeling of 
the parameters $(\alpha_1,\ldots,\alpha_{q+1})$ appearing in 
Corollary~\ref{cor:rho-general}, so that the final form of
Equation~\eqref{eq.qn-any} attains a fully symmetric 
and canonical structure with respect to its summation indices.
This identity is purely combinatorial in nature and holds for all integers 
$q,n\ge0$, independent of any convergence or analytic assumptions.


Since the explicit values of $\rho(\alpha_1+1,\ldots,\alpha_r+1)$ 
can be computed explicitly from Proposition~\ref{prop.rho-value}, 
we record several representative evaluations below.

\begin{enumerate}
\item 
For nonnegative integers $p$ and $\alpha_1,\ldots,\alpha_r$ satisfying $\alpha_r\ge1$, 
\begin{equation} \label{eq.rho-head}
\rho(\{1\}^p,\alpha_1+1,\alpha_2+1,\ldots,\alpha_r+1)
=|\bm{\alpha}|^{-p}\,
 \rho(\alpha_1+1,\alpha_2+1,\ldots,\alpha_r+1).
\end{equation}
In particular, when $r=1$ and $\alpha_1=m-1$, this gives 
\[
\rho(\{1\}^p,m+1) = \frac{1}{m^{p+1}\,m!}.
\]

\item 
To illustrate the structural patterns of multiple $\rho$-values, 
we list several explicit evaluations and closed forms below.

\begin{center}
\captionof{table}{Values of $\rho(\alpha_1,\alpha_2,\ldots,\alpha_r)$ for weights $2\le w\le6$.}
\end{center}
\vskip -0.7cm

\setlength{\columnseprule}{0.4pt} 
\setlength{\columnsep}{20pt}      
\vskip -0.7cm
\begin{multicols}{3}
\TableHeadert
\Rowt{$2$}{$\left(2\right)$}{$1$}
\Rowt{$3$}{$\left(3\right)$}{$1/4$}
\Rowt{$3$}{$\left(1,\,2\right)$}{$1$}
\Rowt{$4$}{$\left(4\right)$}{$1/18$}
\Rowt{$4$}{$\left(1,\,3\right)$}{$1/8$}
\Rowt{$4$}{$\left(2,\,2\right)$}{$1/4$}
\Rowt{$4$}{$\left(1,\,1,\,2\right)$}{$1$}
\Rowt{$5$}{$\left(5\right)$}{$1/96$}
\Rowt{$5$}{$\left(1,\,4\right)$}{$1/54$}
\Rowt{$5$}{$\left(2,\,3\right)$}{$1/36$}
\TableHeadert
\Rowt{$5$}{$\left(3,\,2\right)$}{$1/18$}
\Rowt{$5$}{$\left(1,\,1,\,3\right)$}{$1/16$}
\Rowt{$5$}{$\left(1,\,2,\,2\right)$}{$1/8$}
\Rowt{$5$}{$\left(2,\,1,\,2\right)$}{$1/4$}
\Rowt{$5$}{$\left(1,\,1,\,1,\,2\right)$}{$1$}
\Rowt{$6$}{$\left(6\right)$}{$1/600$}
\Rowt{$6$}{$\left(1,\,5\right)$}{$1/384$}
\Rowt{$6$}{$\left(2,\,4\right)$}{$1/288$}
\Rowt{$6$}{$\left(3,\,3\right)$}{$1/192$}
\Rowt{$6$}{$\left(4,\,2\right)$}{$1/96$}
\TableHeadert
\Rowt{$6$}{$\left(1,\,1,\,4\right)$}{$1/162$}
\Rowt{$6$}{$\left(1,\,2,\,3\right)$}{$1/108$}
\Rowt{$6$}{$\left(1,\,3,\,2\right)$}{$1/54$}
\Rowt{$6$}{$\left(2,\,1,\,3\right)$}{$1/72$}
\Rowt{$6$}{$\left(2,\,2,\,2\right)$}{$1/36$}
\Rowt{$6$}{$\left(3,\,1,\,2\right)$}{$1/18$}
\Rowt{$6$}{$\left(1,\,1,\,1,\,3\right)$}{$1/32$}
\Rowt{$6$}{$\left(1,\,1,\,2,\,2\right)$}{$1/16$}
\Rowt{$6$}{$\left(1,\,2,\,1,\,2\right)$}{$1/8$}
\Rowt{$6$}{$\left(2,\,1,\,1,\,2\right)$}{$1/4$}
\Rowt{$6$}{$\left(1,\,1,\,1,\,1,\,2\right)$}{$1$}
\end{multicols}

For the case of uniform indices, the function admits a compact general formula:
\[
\rho(\{a+1\}^n)
   = \frac{1}{a^n\,n!\,(na)!}.
\]
Setting $a=1$ yields the simple case $\rho(\{2\}^n)=1/(n!)^2$.

For alternating parameters, the function admits a compact one-parameter family of identities:
\[
\rho(\{1,a+1\}^n)
 = \frac{1}{a^{2n}(na)!\,n!^2}.
\]
In particular, setting $a=2$ reproduces the earlier formula
\[
\rho(\{1,3\}^n)
 = \frac{1}{4^n(n!)^2(2n)!}.
\]

Finally, the strictly increasing index pattern can be obtained directly from 
Proposition~\ref{prop.rho-value}:
\[
\rho(1,2,\ldots,n)
=\frac{2^{n+1}(2n-1)}{(n-1)(2n)!\left(\frac{n(n-1)}{2}\right)!}.
\]
\end{enumerate}

\section{Multiple $\eta$ values}

Let $\mathbb{Q}[[x_1,x_2,\ldots]]$ denote the algebra of formal power series 
in countably many variables, each of bounded total degree.  
An element $u\in \mathbb{Q}[[x_1,x_2,\ldots]]$ is called 
\emph{quasisymmetric} if, for every monomial 
$x_{a_1}^{\alpha_1}\cdots x_{a_k}^{\alpha_k}$ with $a_1<a_2<\cdots<a_k$, 
the coefficient of this monomial in $u$ is identical to that of 
$x_1^{\alpha_1}\cdots x_k^{\alpha_k}$.  
The collection of all such elements forms a graded connected Hopf algebra, 
denoted by $\QSym$ (see, for example, \cite{Hof1, Mac, Sta}).

For a composition $I=(\alpha_1,\ldots,\alpha_k)$ of a positive integer, 
the integer $k$ is called the \emph{length} of $I$, and the sum 
$|I|=\alpha_1+\cdots+\alpha_k$ is its \emph{weight}.  
The \emph{monomial quasisymmetric function} associated with $I$ is defined by
\[
   M_I(x_1,x_2,\ldots)
   := \sum_{1\leq i_1<i_2<\cdots<i_k}
      x_{i_1}^{\alpha_1}\cdots x_{i_k}^{\alpha_k},
\]
with the convention $M_0=1$.  
The family $\{M_I\}$ forms a linear basis of $\QSym$.  

\medskip

Fix positive integers $n,m$.  
We introduce the specialization
\[
   x_i = \frac{1}{(i+z)^m}, \quad 1\leq i\leq n,
   \qquad x_i=0, \quad i>n,
\]
where $z\in \mathbb{C}\setminus\{-1,-2,\ldots\}$ so that no denominator vanishes.  
Since only finitely many variables are nonzero under this specialization, 
each $M_I$ reduces to a finite sum, namely
\[
   M_{(\alpha_1,\ldots,\alpha_k)}
   = \zeta_n(m\alpha_1,m\alpha_2,\ldots,m\alpha_k;z),
\]
where $\zeta_n(s_1,\ldots,s_k;z)$ denotes a truncated multiple Hurwitz zeta sum.  
We denote by $\HQ$ the subalgebra of $\QSym$ spanned by these specializations.  

\medskip

For $u\in\HQ$ and $(s_1,\ldots,s_k)\in\mathbb{N}_0^k$ with 
$s_1+\cdots+s_k \ge 2$, we define
\[
   \eta_{s_1,\ldots,s_k}(u;z)
   := \sum_{n=1}^\infty
      \frac{u}{(n+z)^{s_1}(n+1+z)^{s_2}\cdots(n+k-1+z)^{s_k}},
\]
which defines a complex-valued function on $\HQ$.  

\medskip

In the author's previous work~\cite{Chen202010}, 
the functions $\eta_{s_1,\ldots,s_k}(u;z)$ were expressed explicitly 
as linear combinations of multiple Hurwitz zeta functions and 
special values of the generalized harmonic numbers $H_n^{(m)}(z)$.

\medskip

In the present work, we develop an integral approach to the study of 
$\eta_{s_1,\ldots,s_k}(u;z)$.  
As a first step, we focus on the structurally simpler prototype
\[
   \eta(\alpha_1,\alpha_2,\ldots,\alpha_r)
   := \sum_{n=1}^\infty 
      \frac{1}{n^{\alpha_1}(n+1)^{\alpha_2}\cdots (n+r-1)^{\alpha_r}},
\]
which may be viewed as the specialization of the general $\eta$-function 
corresponding to $u=1$ and $z=0$.  
This prototype retains the essential shifted-denominator structure, 
and, as will be shown below, its summations can be connected to those of 
the multiple $\rho$-function through suitable integral transformations.

The quantity $\eta(\alpha_1+1,\alpha_2+1,\ldots,\alpha_r+1)$ 
admits the following integral representation:
\begin{equation}\label{eq.eta-int}
\int_{0<t_1<t_2<\cdots<t_{|\bm\alpha|+r}<1}
\Omega_1\Omega_2\cdots\Omega_{|\bm\alpha|+r},
\end{equation}
where each differential form $\Omega_j$ is determined according to the rule
\[
\Omega_j =
\begin{cases}
\dfrac{dt_1}{1-t_1}, &\text{if } j=1,\\[6pt]
dt_j, &\text{if } j=\alpha_1+2,\,\alpha_1+\alpha_2+3,\,\ldots,\,\alpha_1+\alpha_2+\cdots+\alpha_{r-1}+r,\\[6pt]
\dfrac{dt_j}{t_j}, &\text{otherwise.}
\end{cases}
\]

Consequently, the multiple integral in Equation~\eqref{eq.eta-int} can be rewritten as
\begin{align}
&\eta(\alpha_1+1,\alpha_2+1,\ldots,\alpha_r+1) \\
&\qquad=\frac1{\alpha_1!\alpha_2!\cdots\alpha_r!}\int_{E_r}
  \bigl(\log\tfrac{t_2}{t_1}\bigr)^{\alpha_1}
  \cdots
  \bigl(\log\tfrac{t_r}{t_{r-1}}\bigr)^{\alpha_{r-1}}
  \bigl(\log\tfrac{1}{t_r}\bigr)^{\alpha_r}
  \frac{dt_1\cdots dt_r}{1-t_1}, \nonumber
\end{align}
where the integration domain is the standard $r$-simplex
\[
E_r=\{(t_1,t_2,\ldots,t_r)\in\mathbb{R}^r \mid 0<t_1<t_2<\cdots<t_r<1\}.
\]

\begin{theorem}\label{thm:rho-eta-connection}
For nonnegative integers $r,q$, the following identity holds:
\begin{equation}\label{eq:rho-eta-connection}
\sum_{|\bm{s}|=q}
\eta(s_1+1,s_2+1,\ldots,s_{r+2}+1)
=\sum_{|\bm{\alpha}|=r}
\rho(\alpha_1+1,\alpha_2+1,\ldots,\alpha_q+1,\alpha_{q+1}+2).
\end{equation}
\end{theorem}
\begin{proof}
We first express the left-hand side of Equation~\eqref{eq:rho-eta-connection} 
in the integral form
\[
\sum_{|\bm{\alpha}|=r}
\rho(\alpha_1+1,\alpha_2+1,\ldots,\alpha_q+1,\alpha_{q+1}+2)
=\frac{1}{q!\,r!}
\int_{0<t_1<t_2<1}
\!\!\left(\log\frac{1-t_1}{1-t_2}\right)^{\!q}
(t_2-t_1)^r
\frac{dt_1\,dt_2}{1-t_1}.
\]

Next, we perform the change of variables
\[
u_1=\frac{1-t_2}{1-t_1}, \qquad 
u_2=\frac{1-t_1-t_1t_2+t_1^2}{1-t_1}.
\]
The Jacobian of this transformation is
\[
\frac{\partial(t_1,t_2)}{\partial(u_1,u_2)}
=\frac{u_1-u_2}{(1-u_1)^2}.
\]
After substituting this relation into the integral and simplifying,
the Jacobian factor cancels with the differential terms,
so that the measure again takes the normalized form
\[
\frac{du_1\,du_2}{1-u_1}.
\]

The above integral becomes
\[
\frac{1}{q!\,r!}
\int_{0<u_1<u_2<1}
\left(\log\frac{1}{u_1}\right)^{\!q}
(u_2-u_1)^r
\frac{du_1\,du_2}{1-u_1}.
\]

We then expand the factor 
\[
\left(\log\frac{1}{u_1}\right)^q
=\sum_{a+b=q}\binom{q}{a}
\left(\log\frac{u_2}{u_1}\right)^a
\left(\log\frac{1}{u_2}\right)^b
\]
by the binomial theorem.  
Substituting this expansion into the integral and using 
the integral representation of $\eta$ yields
\[
\sum_{a+b=q}
\sum_{s_1+s_2+\cdots+s_{r+1}=a}
\eta(s_1+1,s_2+1,\ldots,s_{r+1}+1,b+1).
\]
Combining these two summations yields
\[
\sum_{|\bm{s}|=q}
\eta(s_1+1,s_2+1,\ldots,s_{r+2}+1),
\]
which completes the proof.
\end{proof}

The preceding theorem reveals a structural correspondence between 
the multiple $\rho$-sums and the multiple $\eta$-sums: 
through suitable integral substitutions, 
each family of $\rho$-values of depth~$q{+}1$ 
can be transformed into a sum of $\eta$-values 
of greater depth~$r{+}2$ but with the same total weight.  
This correspondence enables one to translate evaluations of 
$\rho$-type series—whose terms involve shifted rising factorials—
into those of $\eta$-type series, 
which are defined by products of shifted powers.  
Such a relation provides an effective analytic bridge 
between the two classes of multiple sums and offers 
a unified perspective on their coincident sum formulas.

In the subsequent results, 
we further exploit this integral framework 
to evaluate several families of $\eta$-sums 
under specific combinatorial constraints, 
expressing their total values in closed form 
through harmonic numbers and Bell-type polynomials.

\begin{theorem}\label{thm.32}
For integers $n\geq 1$ and $q\geq 0$, the following identity holds:
\begin{equation}
\sum_{r+s=n\atop |\bm\alpha|=q}
\eta(\alpha_1+1,\ldots,\alpha_r+1,\alpha_{r+1}+2,\{1\}^s)
=\frac{P_{q+1}(H_n^{(1)},\ldots,H_n^{(q+1)})}{n\cdot n!}.
\end{equation}
\end{theorem}

\begin{proof}
Consider the integral
\[
\frac{1}{n!\,q!}
\int_{0<t_1<t_2<1}
\left(\log\frac{t_2}{t_1}\right)^q
(1-t_1)^n
\frac{dt_1\,dt_2}{(1-t_1)t_2}.
\]
Expanding the factor $(1-t_1)^n$ gives
\[
(1-t_1)^n
=\sum_{r+s=n}\binom{n}{r}(t_2-t_1)^r(1-t_2)^s.
\]
Substituting this expansion into the integral yields
\[
\sum_{r+s=n\atop |\bm\alpha|=q}
\eta(\alpha_1+1,\ldots,\alpha_r+1,\alpha_{r+1}+2,\{1\}^s).
\]

On the other hand, the same integral admits an alternative representation
\[
\frac{1}{n!}
\int_{E_{q+2}}
(1-t_1)^{n-1}\,dt_1
\frac{du_1}{u_1}\cdots \frac{du_q}{u_q}\frac{dt_2}{t_2},
\]
where $E_{q+2}\subset\mathbb{R}^{q+2}$ denotes the $(q+2)$-dimensional ordered simplex
\[
E_{q+2}=\{(t_1,u_1,\ldots,u_q,t_2)\mid 0<t_1<u_1<\cdots<u_q<t_2<1\}.
\]
Evaluating this iterated integral gives
\[
\frac{1}{n!}\sum_{k=0}^{n-1}\binom{n-1}{k}\frac{(-1)^k}{(k+1)^{q+2}}.
\]
Using \cite[Proposition~3.2]{BatirC2019},
\begin{equation}\label{eq.chen}
\sum_{k=0}^n\binom{n}{k}\frac{(-1)^k}{(k+1)^m}
=\frac{P_{m-1}(H_{n+1}^{(1)},\ldots,H_{n+1}^{(m-1)})}{n+1},
\end{equation}
we rewrite the sum as
\[
\frac{P_{q+1}(H_n^{(1)},\ldots,H_n^{(q+1)})}{n\cdot n!},
\]
which completes the proof.
\end{proof}

\medskip

This identity leads directly to a family of elegant summation formulas.
For example:
\begin{itemize}
\item When $q=0$,
\[
\sum_{r+s=n}\eta(\{1\}^r,2,\{1\}^s)
=\frac{H_n}{n\cdot n!}.
\]
\item When $q=1$,
\[
\sum_{r+s=n}\eta(\{1\}^r,3,\{1\}^s)
+\sum_{a+b+s=n-1}\eta(\{1\}^a,2,\{1\}^b,2,\{1\}^s)
=\frac{H_n^2+H_n^{(2)}}{2n\cdot n!}.
\]
\end{itemize}
These examples illustrate how the integral approach systematically produces 
finite combinatorial identities linking multiple $\eta$-values to harmonic-number 
polynomials.

\begin{remark}
After the proof of Theorem~\ref{thm.32}, we observe an intriguing relation among the quantities involved. 
Theorem~\ref{thm.32} shows that 
\[
\sum_{r+s=n\atop |\bm\alpha|=q}
\eta(\alpha_1+1,\ldots,\alpha_r+1,\alpha_{r+1}+2,\{1\}^s)
=\frac{P_{q+1}(H_n^{(1)},\ldots,H_n^{(q+1)})}{n\cdot n!}.
\] 
Moreover, by Theorem~\ref{thm:rho-sum} this common value equals 
\[
\sum_{|\bm{s}|=n-1}\rho(s_1+1,s_2+1,\ldots,s_{q+1}+1,s_{q+2}+2)
\] 
and by Theorem~\ref{thm:rho-eta-connection} we have 
\[
\sum_{|\bm{s}|=n-1}
\rho(s_1+1,s_2+1,\ldots,s_{q+1}+1,s_{q+2}+2)
=\sum_{|\bm{\alpha}|=q+1}
\eta(\alpha_1+1,\alpha_2+1,\ldots,\alpha_{n+1}+1).
\] 
hence 
\[
\sum_{r+s=n\atop |\bm\alpha|=q}
\eta(\alpha_1+1,\ldots,\alpha_r+1,\alpha_{r+1}+2,\{1\}^s) 
= \sum_{|\bm{\alpha}|=q+1}
\eta(\alpha_1+1,\alpha_2+1,\ldots,\alpha_{n+1}+1).
\] 
Although this equality can also be verified directly by rearranging the summation variables, 
this chain of identities highlights a close correspondence between the \(\rho\)-values and the \(\eta\)-values.
\end{remark}

\begin{corollary}\label{cor.34}
For nonnegative integers $n$ and $q$, 
the sum $\sum_{|\bm{\alpha}|=q+1}\eta(\alpha_1+1,\ldots,\alpha_{n+1}+1)$ 
admits the integral representation given in Equation~\eqref{eq.int}.
\end{corollary}

\medskip
By applying the same integral technique as in Theorem~\ref{thm.32},
but replacing the factor $\tfrac{dt_1\,dt_2}{(1-t_1)t_2}$ by 
$\tfrac{dt_1\,dt_2}{1-t_1}$ in the integrand, 
we obtain a weighted variant of the preceding identity.
This yields the following proposition.

\begin{proposition}\label{prop.35}
For integers $n\geq 1$ and $q\geq 0$, we have
\[
\sum_{r+s=n\atop |\bm\alpha|=q}
\eta(\alpha_1+1,\ldots,\alpha_{r+1}+1,\{1\}^{s+1})
=\frac{(-1)^{q+1}}{n\,(n+1)!}
+\frac{1}{n\cdot n!}
\sum_{k=0}^q(-1)^{q-k}
P_k(H_n^{(1)},\ldots,H_n^{(k)}).
\]
\end{proposition}

\begin{proof}[Sketch of proof]
Starting from the integral
\[
\frac{1}{n!\,q!}
\int_{0<t_1<t_2<1}
\left(\log\frac{t_2}{t_1}\right)^q(1-t_1)^n
\frac{dt_1\,dt_2}{1-t_1},
\]
and following the same expansion as in the proof of 
Theorem~\ref{thm.32}, one obtains
\[
\frac{1}{n!}\sum_{k=0}^{n-1}
\binom{n-1}{k}\frac{(-1)^k}{(k+1)^{q+1}(k+2)}.
\]
Decomposing the rational factor as
\[
\frac{1}{(k+1)^{q+1}(k+2)}
=\frac{(-1)^{q+1}}{k+2}
+\sum_{\ell=0}^q
\frac{(-1)^{q-\ell}}{(k+1)^{\ell+1}},
\]
and using Equation~\eqref{eq.chen}, 
the stated formula follows after straightforward simplification.
\end{proof}

\medskip

This result leads to several weighted summation identities.  
For example:
\begin{itemize}
\item When $q=1$,
\begin{equation}\label{eq.w121}
\sum_{a+b=n}(b+1)\,
\eta(\{1\}^a,2,\{1\}^{b+1})
=\frac{(n+1)H_n-n}{n\cdot (n+1)!}.
\end{equation}
\item When $q=2$,
\begin{align}\label{eq.w122}
&\sum_{a+b=n}(b+1)\,
\eta(\{1\}^a,3,\{1\}^{b+1})
+\sum_{a+b+c=n-1}(c+1)\,
\eta(\{1\}^a,2,\{1\}^b,2,\{1\}^{c+1}) \nonumber\\
&\qquad
=\frac{2n+(n+1)\bigl(H_n^2-2H_n+H_n^{(2)}\bigr)}{2n\cdot (n+1)!}.
\end{align}
\item When $n=1$,
\begin{equation}\label{eq.38}
\sum_{\alpha_1+\alpha_2=q}\eta(\alpha_1+1,\alpha_2+1,1)+\eta(q+1,1,1)=\frac12.
\end{equation}
\end{itemize}

These weighted formulas further illustrate the analytic flexibility 
of the integral approach, showing that even when the summands are 
asymmetrically weighted by depth or position, 
the total sum remains expressible in closed form 
through harmonic numbers and Bell-type polynomials.

\medskip
The results obtained in this section demonstrate that 
the multiple $\eta$-values admit a rich family of summation identities 
that can be derived systematically through integral transformations.  
By linking the multiple $\rho$- and $\eta$-functions within the same 
integral framework, we established a precise correspondence between 
their respective sum formulas, allowing the evaluation of otherwise 
intractable $\eta$-sums in terms of harmonic numbers and 
Bell-type polynomials.  

These identities not only reveal the combinatorial nature of the 
multiple $\eta$-function but also indicate a deeper parallelism between 
zeta-like series defined by products of rising factorials 
and those defined by shifted powers.  
This structural connection underlies the coincidence phenomena 
highlighted in the concluding remarks.

\medskip
To conclude this section, we present several explicit evaluations of 
$\eta(\alpha_1+1,\ldots,\alpha_r+1)$ for small total weights, 
illustrating the structural resemblance between $\eta$- and $\rho$-values 
established above. 
These results not only provide concrete verification of the preceding 
integral relations but also reveal parallel patterns to the multiple 
$\rho$-values tabulated earlier.

\begin{center}
\captionof{table}{Explicit evaluations of $\eta(a_1,a_2,\ldots,a_r)$ for weights $2\le w\le6$.}
\end{center}

\setlength{\columnseprule}{0.4pt} 
\setlength{\columnsep}{20pt}      
\begin{multicols}{2}
\TableHeader
\Row{2}{$\left(2\right)$}{$\frac{\pi ^2}{6}$}
\Row{2}{$\left(1,\,1\right)$}{$1$}
\Row{3}{$\left(3\right)$}{$\zeta (3)$}
\Row{3}{$\left(1,\,2\right)$}{$2-\frac{\pi ^2}{6}$}
\Row{3}{$\left(2,\,1\right)$}{$\frac{\pi ^2}{6}\allowbreak{}-\,1$}
\Row{3}{$\left(1,\,1,\,1\right)$}{$\frac{1}{4}$}
\Row{4}{$\left(4\right)$}{$\frac{\pi ^4}{90}$}
\Row{4}{$\left(1,\,3\right)$}{$-\zeta (3)+3-\frac{\pi ^2}{6}$}
\Row{4}{$\left(2,\,2\right)$}{$\frac{\pi ^2}{3}\allowbreak{}-\,3$}
\Row{4}{$\left(3,\,1\right)$}{$\zeta (3)+1-\frac{\pi ^2}{6}$}
\Row{4}{$\left(1,\,1,\,2\right)$}{$\frac{\pi ^2}{12}\allowbreak{}-\,\frac{3}{4}$}
\Row{4}{$\left(1,\,2,\,1\right)$}{$\frac{7}{4}-\frac{\pi ^2}{6}$}
\Row{4}{$\left(2,\,1,\,1\right)$}{$\frac{\pi ^2}{12}-\frac{5}{8}$}
\TableHeader
\Row{4}{$\left(1,\,1,\,1,\,1\right)$}{$\frac{1}{18}$}
\Row{5}{$\left(5\right)$}{$\zeta (5)$}
\Row{5}{$\left(1,\,4\right)$}{$-\zeta (3)+4-\frac{\pi ^2}{6}-\frac{\pi ^4}{90}$}
\Row{5}{$\left(2,\,3\right)$}{$\zeta (3)-6+\frac{\pi ^2}{2}$}
\Row{5}{$\left(3,\,2\right)$}{$\zeta (3)+4-\frac{\pi ^2}{2}$}
\Row{5}{$\left(4,\,1\right)$}{$-\zeta (3)-1+\frac{\pi ^2}{6}+\frac{\pi ^4}{90}$}
\Row{5}{$\left(1,\,1,\,3\right)$}{$\frac{1}{2}\,\zeta (3)\allowbreak{}-\,\frac{29}{16}\allowbreak{}+\,\frac{\pi ^2}{8}$}
\Row{5}{$\left(1,\,2,\,2\right)$}{$\frac{5}{2}-\frac{\pi ^2}{4}$}
\Row{5}{$\left(1,\,3,\,1\right)$}{$\frac{5}{4}-\zeta (3)$}
\Row{5}{$\left(2,\,1,\,2\right)$}{$\frac{1}{16}$}
\Row{5}{$\left(2,\,2,\,1\right)$}{$\frac{\pi ^2}{4}\,\allowbreak{}-\,\frac{19}{8}$}
\Row{5}{$\left(3,\,1,\,1\right)$}{$\frac{1}{2}\,\zeta (3)\allowbreak{}+\,\frac{13}{16}\allowbreak{}-\,\frac{\pi ^2}{8}$}
\Row{5}{$\left(1,\,1,\,1,\,2\right)$}{$\frac{31}{108}\allowbreak{}-\,\frac{\pi ^2}{36}$}
\TableHeader
\Row{5}{$\left(1,\,1,\,2,\,1\right)$}{$\frac{\pi ^2}{12}\allowbreak{}-\,\frac{29}{36}$}
\Row{5}{$\left(1,\,2,\,1,\,1\right)$}{$\frac{61}{72}\allowbreak{}-\,\frac{\pi ^2}{12}$}
\Row{5}{$\left(2,\,1,\,1,\,1\right)$}{$\frac{\pi ^2}{36}\allowbreak{}-\,\frac{49}{216}$}
\Row{5}{$\left(1,\,1,\,1,\,1,\,1\right)$}{$\frac{1}{96}$}
\Row{6}{$\left(6\right)$}{$\frac{\pi ^6}{945}$}
\Row{6}{$\left(1,\,5\right)$}{$-\zeta (3)-\zeta (5)+5-\frac{\pi ^2}{6}-\frac{\pi ^4}{90}$}
\Row{6}{$\left(2,\,4\right)$}{$2 \zeta (3)-10+\frac{2 \pi ^2}{3}+\frac{\pi ^4}{90}$}
\Row{6}{$\left(3,\,3\right)$}{$10-\pi ^2$}
\Row{6}{$\left(4,\,2\right)$}{$-2\,\zeta (3)-5+\frac{2 \pi ^2}{3}+\frac{\pi ^4}{90}$}
\Row{6}{$\left(5,\,1\right)$}{$\zeta (3)+\zeta (5)+1-\frac{\pi ^2}{6}-\frac{\pi ^4}{90}$}
\Row{6}{$\left(1,\,1,\,4\right)$}{$\frac{3}{4}\,\zeta (3)-\frac{23}{8}+\frac{7 \pi ^2}{48}+\frac{\pi ^4}{180}$}
\Row{6}{$\left(1,\,2,\,3\right)$}{$-\,\frac{1}{2}\,\zeta (3)\allowbreak{}+\,\frac{69}{16}\allowbreak{}-\,\frac{3}{8}\,\pi ^2$}
\Row{6}{$\left(1,\,3,\,2\right)$}{$-\zeta (3)\allowbreak{}-\,\frac{5}{4}\allowbreak{}+\,\frac{\pi ^2}{4}$}
\Row{6}{$\left(1,\,4,\,1\right)$}{$\frac{11}{4}-\frac{\pi ^2}{6}-\frac{\pi ^4}{90}$}
\Row{6}{$\left(2,\,1,\,3\right)$}{$-\,\frac{1}{4}\,\zeta (3)\allowbreak{}+\,\frac{15}{16}\allowbreak{}-\,\frac{\pi ^2}{16}$}
\Row{6}{$\left(2,\,2,\,2\right)$}{$\frac{\pi ^2}{4}\allowbreak{}-\,\frac{39}{16}$}
\Row{6}{$\left(2,\,3,\,1\right)$}{$\zeta (3)-\frac{29}{8}+\frac{\pi ^2}{4}$}
\Row{6}{$\left(3,\,1,\,2\right)$}{$\frac{1}{4}\,\zeta (3)\allowbreak{}+\,\frac{3}{8}\allowbreak{}-\,\frac{\pi ^2}{16}$}
\TableHeader
\Row{6}{$\left(3,\,2,\,1\right)$}{$\frac{1}{2}\,\zeta (3)\allowbreak{}+\,\frac{51}{16}\allowbreak{}-\,\frac{3\pi ^2}{8}$}
\Row{6}{$\left(4,\,1,\,1\right)$}{$-\frac{3}{4}\, \zeta (3)-\frac{29}{32}+\frac{7 \pi ^2}{48}+\frac{\pi ^4}{180}$}
\Row{6}{$\left(1,\,1,\,1,\,3\right)$}{$-\,\frac{1}{6}\,\zeta (3)\allowbreak{}+\,\frac{305}{432}\allowbreak{}-\,\frac{11\pi ^2}{216}$}
\Row{6}{$\left(1,\,1,\,2,\,2\right)$}{$\frac{\pi ^2}{9}\allowbreak{}-\,\frac{59}{54}$}
\Row{6}{$\left(1,\,1,\,3,\,1\right)$}{$\frac{1}{2}\,\zeta (3)\allowbreak{}-\,\frac{145}{144}\allowbreak{}+\,\frac{\pi ^2}{24}$}
\Row{6}{$\left(1,\,2,\,1,\,2\right)$}{$\frac{121}{432}\allowbreak{}-\,\frac{\pi ^2}{36}$}
\Row{6}{$\left(1,\,2,\,2,\,1\right)$}{$\frac{119}{72}-\frac{\pi ^2}{6}$}
\Row{6}{$\left(1,\,3,\,1,\,1\right)$}{$-\,\frac{1}{2}\,\zeta (3)\allowbreak{}+\,\frac{29}{144}\allowbreak{}+\,\frac{\pi ^2}{24}$}
\Row{6}{$\left(2,\,1,\,1,\,2\right)$}{$\frac{\pi ^2}{54}\allowbreak{}-\,\frac{37}{216}$}
\Row{6}{$\left(2,\,1,\,2,\,1\right)$}{$\frac{125}{432}\allowbreak{}-\,\frac{\pi ^2}{36}$}
\Row{6}{$\left(2,\,2,\,1,\,1\right)$}{$\frac{\pi ^2}{9}\allowbreak{}-\,\frac{29}{27}$}
\Row{6}{$\left(3,\,1,\,1,\,1\right)$}{$\frac{1}{6}\,\zeta (3)\allowbreak{}+\,\frac{449}{1296}\allowbreak{}-\,\frac{11 \pi ^2}{216}$}
\Row{6}{$\left(1,\,1,\,1,\,1,\,2\right)$}{$\frac{ \pi ^2}{144}\allowbreak{}-\,\frac{115}{1728}$}
\Row{6}{$\left(1,\,1,\,1,\,2,\,1\right)$}{$\frac{239}{864}\allowbreak{}-\,\frac{\pi ^2}{36}$}
\Row{6}{$\left(1,\,1,\,2,\,1,\,1\right)$}{$\frac{\pi ^2}{24}\allowbreak{}-\,\frac{235}{576}$}
\Row{6}{$\left(1,\,2,\,1,\,1,\,1\right)$}{$\frac{241}{864}\allowbreak{}-\,\frac{\pi ^2}{36}$}
\Row{6}{$\left(2,\,1,\,1,\,1,\,1\right)$}{$\frac{\pi ^2}{144}\allowbreak{}-\,\frac{205}{3456}$}
\Row{6}{$\left(1,\,1,\,1,\,1,\,1,\,1\right)$}{$\frac{1}{600}$}
\end{multicols}

\section{Concluding Remarks}

This paper has investigated two interrelated families of zeta-like multiple series:
the multiple $\rho$-function and the multiple $\eta$-function.
Both are defined by nested sums with shifted denominators,
and their structural correspondence has been clarified through
a combination of factorial, integral, and combinatorial analyses.

\medskip
\noindent
(1) For the multiple $\rho$-values, we established a closed factorial formula
(Proposition~\ref{prop.rho-value}) and a general summation relation connecting
$\rho$ to finite multiple zeta-star values (Theorem~\ref{thm:rho-sum}).
These results reveal the factorial structure intrinsic to the $\rho$-family
and serve as the foundation for various weighted and unweighted sum identities.
A notable outcome is the combinatorial identity
(Equation~\eqref{eq.qn-any}), which holds for all $q,n\ge0$
and expresses a perfect balance among discrete compositions of~$n$.
This identity captures the intrinsic combinatorial pattern underlying the $\rho$-family, 
reflecting the discrete structure implicit in its weighted-sum relations.

\medskip
\noindent
(2) For the multiple $\eta$-values, an integral framework was developed
that links their summations directly to those of $\rho$
through a precise change of variables.
Both families share the same total weight $q+r+1$,
with depths $\ell(\bm{s})=r+1$ and $\ell(\bm{\alpha})=q+1$, respectively.
Their integral correspondence (Corollary \ref{cor.34}), derived from Theorem \ref{thm:rho-eta-connection}, provides a continuous analogue of the discrete structure underlying $\rho$.
Furthermore, several weighted sum relations for $\eta$ are obtained independently through Proposition \ref{prop.35}, which yields the identities in Equation~\eqref{eq.w121} and Equation~\eqref{eq.w122}.
These results demonstrate how factorial and harmonic formulations
emerge as two aspects of a unified analytic–combinatorial system.

\medskip
\noindent
(3) Following the approach of~\cite{Chen202010},
the explicit evaluation of $\eta(p,\{1\}^a)$ is obtained as
\begin{align*}
a!\cdot \eta(p,\{1\}^a)
 &= \sum_{n=1}^\infty \frac{1}{n^p \binom{n+a}{a}}\\
 &= \sum_{k=0}^{p-2} (-1)^k \zeta(p-k)
     P_k(H_a^{(1)},\ldots,H_a^{(k)})
     + \sum_{k=0}^{a-1} \binom{a}{k+1}
       \frac{(-1)^{p+k-1}}{(k+1)^{p-1}} H_{k+1}.
\end{align*}
Substituting this expression into Equation~\eqref{eq.38}
yields a restricted sum formula for
\begin{align*}
\sum_{\alpha_1+\alpha_2=q}
   \eta(\alpha_1+1,\alpha_2+1,1)
   &=(-1)^{q+1}+\frac{1}{2}
     +\frac{(-1)^q\!\cdot\!3}{2^{q+2}}  \\
   &\quad
     +\sum_{k=0}^{q-1}(-1)^{k+1}
      \!\left(1-\frac{1}{2^{k+1}}\right)\!
      \zeta(q+1-k),
\end{align*}
which gives an explicit evaluation for this special family of triple $\eta$-values.
This result refines the analytic evaluation of $\eta$-values obtained in Section 3,
illustrating how explicit analytic expressions can lead to restricted combinatorial sum formulas.

\medskip
\noindent
(4) Conceptually, the integral method introduced in this work reveals
a duality between the iterated-sum form of $\rho$-values
and the shifted-denominator form of $\eta$-values.
Together they constitute a coherent framework for zeta-like multiple series,
bridging factorial and harmonic perspectives within a unified analytic scheme.

\medskip
\noindent
The techniques developed here open several directions for future research.
Potential extensions include the study of $q$-analogues and Hurwitz-type
deformations of $\rho$ and $\eta$, their connections with Schur multiple
zeta functions, and the algebraic relations arising from shuffle and stuffle
structures among these generalized series.

\medskip
\noindent
In conclusion, the multiple $\rho$- and $\eta$-values
represent two complementary realizations of a deeper combinatorial symmetry
governing zeta-like multiple sums.
Their interplay through integral representations provides
a unified analytic–combinatorial perspective
and lays a promising foundation for further investigation
into the algebraic and analytic properties of multiple zeta-type functions.

\titleformat{\section}
{\sffamily\color{sectitlecolor}\Large\bfseries\filcenter}{}{2em}{#1}%

\subsection*{Acknowledgement.}
The author would like to express his sincere gratitude to Professor Minking Eie for valuable discussions and inspiring comments during the preparation of this work.

\medskip
\subsection*{Funding.}
This research was funded by the National Science and Technology Council, 
Taiwan, R. O. C., under grant number NSTC 114-2115-M-845-001.

\medskip
\subsection*{Conflicts of interest.}
The author declares no conflict of interest.

\medskip
\subsection*{Data Availability Statement.}
The data supporting the findings of this study are contained within the article. No additional data are available.

\medskip
\subsection*{Use of AI Tools.}
The author used OpenAI’s ChatGPT to assist in language editing and improving the clarity of the manuscript. All mathematical content, proofs, and results are entirely the author's own work.

\end{document}